\documentclass[a4paper,11pt]{article}
\usepackage{mathrsfs, amsmath, amsxtra, amssymb, latexsym, amscd, amsthm}

\newtheorem{thm}{Theorem}[section]
\newtheorem{cor}{Corollary}[section]
\newtheorem{lem}{Lemma}[section]

\theoremstyle{definition}
\newtheorem{defn}{Definition}[section]

\theoremstyle{remark}

\numberwithin{equation}{section}
\def\ind{{\rm 1\hspace{-0.90ex}1}}
\def\E{\mathbb{E}}
\begin{document}
\title{Total variation bounds in the Lindeberg central limit theorem}

\author{Nguyen Tien Dung\thanks{Department of Mathematics, VNU University of Science, Vietnam National University, Hanoi, 334 Nguyen Trai, Thanh Xuan, Hanoi, 084 Vietnam.}\,\,\footnote{Corresponding author. Email: dung@hus.edu.vn} \and Hoang Thi Phuong Thao$^\ast$}

\date{\today}          
\maketitle
\begin{abstract} In this paper, we obtain an explicit total variation bound  in the central limit theorem for the sums of non-i.i.d. random variables. Our results show that, under suitable assumptions, Lindeberg's condition is sufficient and necessary for the convergence in total variation distance.
\end{abstract}
\section{Introduction}
Let $(X_k)_{k\geq 1}$ be a sequence of independent real-valued random variables (not necessarily identically distributed) with mean $\E X_k=0$ and finite variances $\sigma_k^2=\E X_k^2\in (0,\infty).$ We consider the normalized sums
\begin{equation}\label{dkrt}
S_n=\frac{X_1+\cdots+X_n}{b_n},
\end{equation}
where $b_n=\sqrt{\sigma_1^2+\sigma_2^2+\cdots+\sigma_n^2},\,\,n\geq 1.$

Let us recall the classical central limit theorem: if the random variables $X_ k$ satisfies Lindeberg's condition:
\begin{equation}\label{lin01}
\lim\limits_{n\to\infty}\frac{1}{b_n^2}\sum\limits_{k=1}^n\E[X_k^2\ind_{\{|X_k|>\varepsilon b_n\}}]=0
\end{equation}
for all $\varepsilon>0,$ where $\ind$ is the indicator function, then $S_n$ converges in distribution to a standard normal random variable $N$ as $n\to \infty.$ Furthermore, if the random variables $X_ k$ satisfy the Feller-L\'evy condition
\begin{equation}\label{lin02}
\max\limits_{1\leq k\leq n}(\sigma_k^2/b_n^{2})\to 0\,\,\,\text{as}\,\,n\to\infty,
\end{equation}
then Lindeberg's condition is both sufficient and necessary. We also have the following optimal bound on the Kolmogorov distance for the rate of convergence (see \cite{Bobkov2023} for a review)
\begin{align}
d_K(S_n,N) := \sup_{x \in \mathbb{R}}|P(S_n\leq x) - P(N\leq x)|&\leq \frac{c}{b_n^3}\sum\limits_{k=1}^n\E[|X_k|^2(b_n\wedge |X_k|)]\label{lin03}\\
&\leq \frac{c}{b_n^3}\sum\limits_{k=1}^n\E|X_k|^3,\label{lin04}
\end{align}
where $c$ is a positive absolute constant. We now consider the total variation distance between the laws of $S_n$ and $N$ defined by
$$d_{TV}(S_n,N):= \frac{1}{2}\sup_{\|h\|_\infty\le 1}|E h(S_n)- E h(N)|,$$
where $h$ is measurable and $\|h\|_\infty:=\sup\limits_{x\in \mathbb{R}}|h(x)|.$  Note that the convergence in total variation is stronger than the convergence in distribution. The study of the central limit theorem in total variation distance has a long history beginning in 1952 with the results of Prokhorov \cite{Prokhorov52}. However, the problem of obtaining optimal estimates for $d_{TV}(S_n,N)$ (under minimal moment conditions) is non-trivial. The reader can consults \cite{Bally2016,Sirazdinov62}  and references therein for the case of independent identically distributed (i.i.d.) random variables with finite third absolute moment. For the case of non-i.i.d. random variables, it is natural to wonder whether the bounds (\ref{lin03}) and (\ref{lin04}) may be strengthened to the total variation distance as well.

When all summands $S_n$ have discrete distributions, $d_{TV}(S_n,N)=1\nrightarrow 0$ as $n\to\infty.$ Hence, to obtain an affirmative answer, we need to impose additional assumptions on the sequence $(X_k)_{k\geq 1}.$ We have the following result, due to Bobkov-Chistyakov-G\"otze \cite{Bobkov2014a}: Assume that the independent random variables $X_1,\cdots , X_n$ have finite third absolute moments, and that $D(X_k)\leq D$ for all $1\leq k\leq n$ and for some $D>0.$ Then
$$d_{TV}(S_n,N)\leq \frac{C_D}{b_n^3}\sum\limits_{k=1}^n\E|X_k|^3,$$
where $C_D$ is a positive constant depending only on $D.$ Here $D(.)$ denotes the relative entropy defined as follows: for a random variable $X$ with density $p(x)$ and finite variance,
$$\text{$D(X)=\int_{-\infty}^\infty p(x)\log \frac{p(x)}{\varphi_X(x)}dx,$ where $\varphi_X(x)=\frac{1}{\sqrt{2\pi{\rm Var}(X)}}e^{-\frac{(x-\E X)^2}{2{\rm Var}(X)}},x\in \mathbb{R}.$}$$
The aim of this paper is to establish the bound (\ref{lin03}) for $d_{TV}(S_n,N).$ Standing for the finiteness of relative entropies, we require the random variables $X_k's$ to have the finite standardised Fisher information. We recall that, given a random variable $X$ with an absolutely continuous density $p(x),$ the standardised Fisher information of $X$ (or its distribution) is defined by
$$
J_{st}(X)= {\rm Var}(X)\int_{-\infty}^{+\infty} \frac{p'(x)^2} {p(x)}dx-1={\rm Var}(X)\E[\rho^2(X)]-1,
$$
where $p'$ denotes a Radon-Nikodym derivative of $p$  and $\rho:=p'/p$ is the score function.  For more details about the Fisher information and the central limit theorem results, we refer the reader to the monograph \cite{Johnson2004} and recent papers \cite{Bobkov2014b,Dung2024,Johnson2020}.

Hereafter, we put $J(X)=J_{st}(X)+1.$ Our main results are stated in the following theorem.
\begin{thm}\label{plu}Let $(X_k)_{k\geq 1}$ be a sequence of independent real-valued random variables $\E X_k=0$ and finite variances $\sigma_k^2=\E X_k^2\in (0,\infty).$ Assume that, for every $k\geq 1,$ the law of $X_k$ has an absolutely continuous density $p_k$ satisfying $J(X_k)<\infty.$ Then, for $S_n$ defined by (\ref{dkrt}), we have
\begin{equation}\label{vhas}d_{TV}(S_n,N)\leq \bigg(\frac{8\pi\max\limits_{1\leq k\leq n}J(X_k)}{1-\max\limits_{1\leq k\leq n}(\sigma_k^2/b_n^{2})}\bigg)^{\frac{1}{2}}\frac{\sum\limits_{k=1}^n\E[|X_k|^2(b_n\wedge |X_k|)]}{b_n^3}.
\end{equation}
\end{thm}
Our condition $J(X_k)<\infty$ is stronger than the condition $D(X_k)<\infty$ required in \cite{Bobkov2014a}. Indeed, we have $D(X_k)\leq(J(X_k)-1)/2,$ see e.g. Lemma E.3 in \cite{Johnson2004}. However, our bound (\ref{vhas}) allows us to upgrade the Lindeberg central limit theorem to the convergence in total variation distance as follows.
\begin{cor}\label{plu1f}Let the sequence $(X_k)_{k\geq 1}$ be as in Theorem \ref{plu}. Assume additionally that $\sup\limits_{k\geq 1}J(X_k)<\infty.$ The, the following two statements are equivalent

\noindent $(i)$ Lindeberg's condition (\ref{lin01}) is satisfied.

\noindent $(ii)$ $d_{TV}(S_n,N)\to 0$ as $n\to\infty$ and the Feller-L\'evy condition (\ref{lin02}) is satisfied.

\end{cor}
We use the condition $J(X_k)<\infty$ is to establish the bound (\ref{vhas}). However, this condition can be removed to recover the classical central limit theorem.
\begin{cor}\label{p30lu1f}Let $(Y_k)_{k\geq 1}$ be a sequence of independent real-valued random variables $\E Y_k=0$ and finite variances $\alpha_k^2=\E Y_k^2\in (0,\infty).$ Assume that this sequence satisfies Lindeberg's condition:
\begin{equation}\label{lin01p}
\lim\limits_{n\to\infty}\frac{1}{a_n^2}\sum\limits_{k=1}^n\E[Y_k^2\ind_{\{|Y_k|>\varepsilon a_n\}}]=0
\end{equation}
for all $\varepsilon>0,$  where $a_n:=\sqrt{\alpha_1^2+\cdots+\alpha_n^2}.$ Then, as $n\to\infty,$ we have
$\frac{Y_1+\cdots+Y_n}{a_n}\to N$ in distribution.
\end{cor}

\section{Proofs}
We first prepare some technical results. For a given real valued measurable function $h$ with $\E h(Z)<\infty.$ The ordinary differential equation
      \begin{equation} \label{stein}
      f'(x) - xf(x) = h(x) - \mathbb{E}h(Z)
      \end{equation}
is called Stein's equation associated with $h$. The unique bounded solution of (\ref{stein}) is given by
\begin{equation} \label{steintv}
f(x):= e^{x^2/2}\int_{-\infty}^x (h(y) - \mathbb{E}h(Z))e^{-y^2/2} dy,\,\,x\in \mathbb{R}.
\end{equation}
We recall from Lemma 2.4 in \cite{Louischen2011} the following estimates for $f.$
\begin{lem}\label{s7ervc}  Let $f$ be the solution (\ref{steintv}) to Stein's equation. If $h$ is bounded $1,$ then
\begin{equation*}
\text{$\|f\|_\infty\leq \sqrt{2\pi}$ and $\|f'\|_\infty\leq 4$}.
\end{equation*}
\end{lem}
\begin{lem}\label{servc} Let the sequence $(X_k)_{k\geq 1}$ be as in Theorem \ref{plu}. For each $1\leq k\leq n,$ we put
$$S_{k,n}=\frac{X_1+\cdots+X_{k-1}+X_{k+1}+\cdots+X_n}{b_{k,n}},$$
where $b_{k,n}=\sqrt{b_n^2-\sigma_k^2}.$ Then the law of $S_{k,n}$ has an absolutely continuous density. Moreover, the score function $\rho_{k,n}$ of $S_{k,n}$ is given by
\begin{equation}\label{8dls}
\rho_{k,n}(x)=\E\left[\frac{\sum\limits_{1\leq i\leq n,i\neq k}\sigma_i^2\rho_i(X_i)}{b_{k,n}}\bigg| S_{k,n}=x\right],
\end{equation}
where $\rho_i$ denotes the score function of $X_i.$
\end{lem}
\begin{proof} For each $k\geq 1,$ it follows from Theorem VIII.1 in \cite{Bobkov2019} that the density $p_k$ of the random variable $X_k$ satisfies
$$\lim\limits_{x\to \pm\infty}(1+|x|)p_k(x)=0.$$
Hence, for any bounded and differentiable function $f:\mathbb{R}\to \mathbb{R}$ with bounded derivative, we can use Lemma 1.20 in \cite{Johnson2004} to get
$$\E[f(X_k)\rho_k(X_k)]=-\E[f'(X_k)].$$
Consequently, by the independence of $X_k's,$ we obtain
\begin{align}
\E\left[f(S_{k,n})\frac{\sum\limits_{1\leq i\leq n,i\neq k}\sigma_i^2\rho_i(X_i)}{b_{k,n}}\right]&=\frac{\sum\limits_{1\leq i\leq n,i\neq k}\sigma_i^2\E\left[f(S_{k,n})\rho_i(X_i)\right]}{b_{k,n}}\notag\\
&=-\E[f'(S_{k,n})].\label{ydlpo}
\end{align}
Then, by using the same arguments as in the proof of Proposition 3.2 in our recent paper \cite{Dung2024}, we conclude that the law of $S_{k,n}$ has an absolutely continuous density and its score function is given by (\ref{8dls}). This finishes the proof of the lemma.
\end{proof}
\begin{defn}\label{klfa} Given a random variable $Y\in L^1(\Omega)$ with a density $p_Y$ supported on some connected interval of $\mathbb{R},$  we define the Stein kernel $\tau_Y$ associated to  $Y$ as follows
$$\tau_Y(x)=\frac{\int_{x}^\infty (y-\E[Y])p_Y(y)dy}{p_Y(x)}$$
for $x$ in the support of $p_Y.$ 
\end{defn} 
In the proofs, we use the following fundamental properties of Stein kernels, see e.g. \cite{Ley2017} and references therein: $\tau_Y(Y)\geq 0,$ $\E[\tau_Y(Y)]={\rm Var}(Y)$ and
\begin{equation}\label{k2hd}
\E[f(Y)(Y-\E Y)]=\E[f'(Y)\tau_Y(Y)]
\end{equation}
for any differentiable function $f:\mathbb{R}\to \mathbb{R}$ with bounded derivative.

We now are ready to prove the main results of the present paper.

\noindent{\bf Proof of Theorem \ref{plu}.} We will carry out the proof in two steps.

\noindent{\it Step 1.} In this step, we additionally assume that
\begin{equation}\label{gye}
\text{\it For every $k\geq 1,$ the density $p_k$ of $X_k$ is supported on some connected interval of $\mathbb{R}.$}
\end{equation}
According to Definition \ref{klfa}, this assumption allows us to define the Stein kernels
$$\tau_k(X_k)=\frac{\int_{X_k}^\infty yp_k(y)dy}{p_k(x)},\,\,k\geq 1.$$
Let $h$ be a measurable function bounded by $1.$ From the Stein equation (\ref{stein}) we have
\begin{align*}
\E h(S_n)-\E h(N)&=\E f'(S_n) - \E[S_nf(S_n)] \\
&=\E f'(S_n) - \frac{1}{b_n}\sum\limits_{k=1}^n\E[X_kf(S_n)].
\end{align*}
Hence, by the property (\ref{k2hd}) of the Stein kernels, we deduce
\begin{align}
\E h(S_n)-\E h(N)&=\E f'(S_n) - \frac{1}{b_n^2}\sum\limits_{k=1}^n\E[\tau_k(X_k)f'(S_n)]\notag\\
&=\frac{1}{b_n^2}\E\left[f'(S_n)(b_n^2-\tau_1(X_1)-\cdots-\tau_n(X_n))\right]\notag\\
&=\frac{1}{b_n^2}\sum\limits_{k=1}^n\E\left[f'(S_n)(\sigma_k^2-\tau_k(X_k))\right].\label{servc8}
\end{align}
Fixed $1\leq k\leq n.$ Let $S_{k,n}$ be as in Lemma \ref{servc}. By the independence of $X_k$ and $S_{k,n},$ we can use the relation (\ref{ydlpo}) to get
\begin{align*}
\E\left[f'(S_n)(\sigma_k^2-\tau_k(X_k))\right]&=\E\left[f'\bigg(\frac{b_{k,n}S_{k,n}+X_k}{b_n}\bigg)(\sigma_k^2-\tau_k(X_k))\right]\\
&=-\frac{b_n}{b_{k,n}}\E\left[f\bigg(\frac{b_{k,n}S_{k,n}+X_k}{b_n}\bigg)(\sigma_k^2-\tau_k(X_k))\rho_{k,n}(S_{k,n})\right].
\end{align*}
On the other hand, since $\E[\tau_k(X_k)]=\sigma_k^2,$ we have
$$\E\left[f\bigg(\frac{b_{k,n}S_{k,n}}{b_n}\bigg)(\sigma_k^2-\tau_k(X_k))\rho_{k,n}(S_{k,n})\right]
=\E\left[f\bigg(\frac{b_{k,n}S_{k,n}}{b_n}\bigg)\rho_{k,n}(S_{k,n})\right]\E[\sigma_k^2-\tau_k(X_k)]=0.$$
So we obtain
\begin{align}
&\E\left[f'(S_n)(\sigma_k^2-\tau_k(X_k))\right]\notag\\
&=-\frac{b_n}{b_{k,n}}\E\left[\bigg(f\bigg(\frac{b_{k,n}S_{k,n}+X_k}{b_n}\bigg)-f\bigg(\frac{b_{k,n}S_{k,n}}{b_n}\bigg)\bigg)
(\sigma_k^2-\tau_k(X_k))\rho_{k,n}(S_{k,n})\right].\label{ttdk01}
\end{align}
Furthermore, we have
\begin{align}
&\bigg|f\bigg(\frac{b_{k,n}S_{k,n}+X_k}{b_n}\bigg)-f\bigg(\frac{b_{k,n}S_{k,n}}{b_n}\bigg)\bigg|\notag\\
&=\bigg|f\bigg(\frac{b_{k,n}S_{k,n}+X_k}{b_n}\bigg)-f\bigg(\frac{b_{k,n}S_{k,n}}{b_n}\bigg)\bigg|\ind_{\{|X_k|>b_n\}}
+\bigg|f\bigg(\frac{b_{k,n}S_{k,n}+X_k}{b_n}\bigg)-f\bigg(\frac{b_{k,n}S_{k,n}}{b_n}\bigg)\bigg|\ind_{\{|X_k|\leq b_n\}}.\notag
\end{align}
Then, by using the estimates $\|f\|_\infty\leq \sqrt{2\pi}$ and $\|f'\|_\infty\leq 4$ from Lemma \ref{s7ervc}, we deduce
\begin{align}
\bigg|f\bigg(\frac{b_{k,n}S_{k,n}+X_k}{b_n}\bigg)-f\bigg(\frac{b_{k,n}S_{k,n}}{b_n}\bigg)\bigg|
&\leq \sqrt{8\pi}\ind_{\{|X_k|>b_n\}}+\frac{4|X_k|}{b_n}\ind_{\{|X_k|\leq b_n\}}\notag\\
&\leq \frac{\sqrt{8\pi}}{b_n}(b_n\wedge |X_k|).\label{ttdk02}
\end{align}
Inserting (\ref{ttdk02}) into (\ref{ttdk01}) yields
\begin{align}
|\E&\left[f'(S_n)(\sigma_k^2-\tau_k(X_k))\right]|\leq
\frac{\sqrt{8\pi}}{b_{k,n}}\E\left[(b_n\wedge |X_k|)|(\sigma_k^2-\tau_k(X_k))\rho_{k,n}(S_{k,n})|\right]\notag\\
&\leq
\frac{\sqrt{8\pi}}{b_{k,n}}\E\left[(b_n\wedge |X_k|)|\sigma_k^2-\tau_k(X_k)|\right]\E|\rho_{k,n}(S_{k,n})|\notag\\
&\leq \frac{\sqrt{8\pi}}{b_{k,n}}\E|\rho_{k,n}(S_{k,n})|\E\left[(b_n\wedge |X_k|)\tau_k(X_k)\right]+\frac{\sqrt{8\pi}}{b_{k,n}}\E|\rho_{k,n}(S_{k,n})|\E\left[(b_n\wedge |X_k|)\sigma_k^2\right].\label{uerjs}
\end{align}
Once again, we use the property (\ref{k2hd}) of the Stein kernels to get the following estimate for the first addend in the right hand side of (\ref{uerjs})
\begin{align*}
\E[(b_n\wedge |X_k|)\tau_k(X_k)]&=\E\left[X_k\int_0^{X_k}(b_n\wedge |y|)dy\right]\\
&\leq\E\left[|X_k|\int_0^{|X_k|}(b_n\wedge y)dy\right]\\
&\leq \E[|X_k|^2(b_n\wedge |X_k|)].
\end{align*}
For the second addend in the right hand side of (\ref{uerjs}), by Chebyshev's association inequality, we have
$$\E[(b_n\wedge |X_k|)\sigma_k^2]=\E[(b_n\wedge |X_k|)]\E[|X_k|^2]\leq  \E[|X_k|^2(b_n\wedge |X_k|)].$$
So it holds that
$$|\E\left[f'(S_n)(\sigma_k^2-\tau_k(X_k))\right]|\leq\frac{2\sqrt{8\pi}}{b_{k,n}}\E|\rho_{k,n}(S_{k,n})|\E[|X_k|^2(b_n\wedge |X_k|)],\,\,1\leq k\leq n.$$
From the above estimates and (\ref{servc8}), we conclude that
\begin{align*}
|\E h(S_n)-\E h(N)|\leq \frac{2\sqrt{8\pi}}{b_n^2}\sum\limits_{k=1}^n\frac{\E|\rho_{k,n}(S_{k,n})|\E[|X_k|^2(b_n\wedge |X_k|)]}{b_{k,n}}
\end{align*}
for any measurable function $h$ bounded by $1.$

\noindent{\it Step 2.} In this step, we remove the assumption (\ref{gye}) and conclude the proof. Let $(N_k)_{k\geq 1}$ be a sequence of independent standard normal random variables, and independence of $(X_k)_{k\geq 1}.$ Fixed $\delta>0,$ we put $\tilde{X}_k=X_k+\delta N_k$  and
$$\tilde{S}_n=\frac{\tilde{X}_1+\cdots+\tilde{X}_n}{\tilde{b}_n},\,\,\,
\tilde{S}_{k,n}=\frac{\tilde{X}_1+\cdots+\tilde{X}_{k-1}+\tilde{X}_{k+1}+\cdots+\tilde{X}_n}{\tilde{b}_{k,n}},$$
where $\tilde{b}_n=\sqrt{\sigma_1^2+\sigma_2^2+\cdots+\sigma_n^2+n\delta^2}$ and $\tilde{b}_{k,n}=\sqrt{\tilde{b}_n^2-\sigma_k^2-\delta^2}.$ Since the sequence $(\tilde{X}_k)_{k\geq 1}$ fulfils the assumption (\ref{gye}), this allows to use the result proved in {\it Step 1} and we  obtain
\begin{align}
|\E h(\tilde{S}_n)-\E h(N)|\leq \frac{2\sqrt{8\pi}}{\tilde{b}_n^2}\sum\limits_{k=1}^n\frac{\E|\tilde{\rho}_{k,n}(\tilde{S}_{k,n})|\E[|\tilde{X}_k|^2(\tilde{b}_n\wedge |\tilde{X}_k|)]}{\tilde{b}_{k,n}}\label{k2hdoo}
\end{align}
for any measurable function $h$ bounded by $1,$ where $\tilde{\rho}_{k,n}$ denotes the score function of $\tilde{S}_{k,n}.$

Obviously, we have $\lim\limits_{\delta\to 0}\tilde{b}_n^2=b_n^2,$ $\lim\limits_{\delta\to 0}\tilde{b}_{k,n}=b_{k,n}$ and $\lim\limits_{\delta\to 0}\E[|\tilde{X}_k|^2(\tilde{b}_n\wedge |\tilde{X}_k|)]=\E[|X_k|^2(b_n\wedge |X_k|)].$  We also have $\lim\limits_{\delta\to 0}\E|\tilde{\rho}_{k,n}(\tilde{S}_{k,n})|)=\E|\rho_{k,n}(S_{k,n})|$ by Theorem 1.4 in \cite{Bobkov2025}. As a consequence, by letting $\delta\to 0,$ we obtain from (\ref{k2hdoo}) that
\begin{align*}
|\E h(S_n)-\E h(N)|\leq \frac{2\sqrt{8\pi}}{b_n^2}\sum\limits_{k=1}^n\frac{\E|\rho_{k,n}(S_{k,n})|\E[|X_k|^2(b_n\wedge |X_k|)]}{b_{k,n}}
\end{align*}
for any continuous function $h$ bounded by $1.$ In addition, from the representation formula (\ref{8dls}), we have
\begin{align*}
\E|\rho_{k,n}(S_{k,n})|&\leq \sqrt{\E|\rho_{k,n}(S_{k,n})|^2}\leq \frac{1}{b_{k,n}}\bigg(\E\bigg|\sum\limits_{1\leq i\leq n,i\neq k}\sigma_i^2\rho_i(X_i)\bigg|^2\bigg)^{\frac{1}{2}}\\
&=\frac{1}{b_{k,n}}\bigg(\sum\limits_{1\leq i\leq n,i\neq k}\sigma_i^4\E|\rho_i(X_i)|^2\bigg)^{\frac{1}{2}}=\frac{1}{b_{k,n}}\bigg(\sum\limits_{1\leq i\leq n,i\neq k}\sigma_i^2J(X_i)\bigg)^{\frac{1}{2}}\\
&\leq \big(\max\limits_{1\leq k\leq n}J(X_k)\big)^{\frac{1}{2}}.
\end{align*}
We therefore obtain, for any continuous function $h$ bounded by $1,$
\begin{align}
|\E h(S_n)-\E h(N)|&\leq \frac{2\big(8\pi\max\limits_{1\leq k\leq n}J(X_k)\big)^{\frac{1}{2}}}{b_n^2}\sum\limits_{k=1}^n\frac{\E[|X_k|^2(b_n\wedge |X_k|)]}{b_{k,n}}\notag\\
&= \frac{2\big(8\pi\max\limits_{1\leq k\leq n}J(X_k)\big)^{\frac{1}{2}}}{b_n^3}\sum\limits_{k=1}^n\frac{\E[|X_k|^2(b_n\wedge |X_k|)]}{\big(1-\sigma_k^2/b_n\big)^{\frac{1}{2}}}\notag\\
&\leq 2\bigg(\frac{8\pi\max\limits_{1\leq k\leq n}J(X_k)}{1-\max\limits_{1\leq k\leq n}(\sigma_k^2/b_n^{2})}\bigg)^{\frac{1}{2}}\frac{\sum\limits_{k=1}^n\E[|X_k|^2(b_n\wedge |X_k|)]}{b_n^3}.\label{vhass}
\end{align}
By an application of Lusin's theorem,
$$d_{TV}(S_n,N)= \frac{1}{2}\sup_{h\in \mathcal{C},\|h\|_\infty\le 1}|E h(S_n)- E h(N)|,$$
where $\mathcal{C}$ denotes the space of continuous functions.  So the bound (\ref{vhas}) follows directly from (\ref{vhass}). The proof of Theorem \ref{plu} is complete. \hfill$\square$



\noindent{\bf Proof of Corollary \ref{plu1f}.} The implication $(ii)\Rightarrow (i)$ is obvious because of the fact that the convergence in total variation implies the convergence in distribution. To verify the implication $(i)\Rightarrow (ii),$ we observe that, for any $\varepsilon\in (0,1),$ we have
\begin{align*}
&\frac{\sum\limits_{k=1}^n\E[|X_k|^2(b_n\wedge |X_k|)]}{b_n^3}=\frac{\sum\limits_{k=1}^n\E[|X_k|^2\ind_{\{|X_k|>b_n\}}]}{b_n^2}+\frac{\sum\limits_{k=1}^n\E[|X_k|^3\ind_{\{|X_k|\leq b_n\}}]}{b_n^3}\\
&=\frac{\sum\limits_{k=1}^n\E[|X_k|^2\ind_{\{|X_k|>b_n\}}]}{b_n^2}+\frac{\sum\limits_{k=1}^n\E[|X_k|^3\ind_{\{\varepsilon b_n<|X_k|\leq b_n\}}]}{b_n^3}+\frac{\sum\limits_{k=1}^n\E[|X_k|^3\ind_{\{|X_k|\leq \varepsilon b_n\}}]}{b_n^3}\\
&\leq \frac{\sum\limits_{k=1}^n\E[|X_k|^2\ind_{\{|X_k|>b_n\}}]}{b_n^2}+\frac{\sum\limits_{k=1}^n\E[|X_k|^2\ind_{\{\varepsilon b_n<|X_k|\leq b_n\}}]}{b_n^2}+\frac{\varepsilon\sum\limits_{k=1}^n\E[|X_k|^3\ind_{\{|X_k|\leq \varepsilon b_n\}}]}{b_n^2}\\
&\leq \frac{\sum\limits_{k=1}^n\E[|X_k|^2\ind_{\{|X_k|>\varepsilon b_n\}}]}{b_n^2}+\varepsilon.
\end{align*}
On the other hand, it is well known that  Lindeberg's condition implies the Feller-L\'evy condition (\ref{lin02}). Hence, recalling the bound (\ref{vhas}), we deduce
$$\lim\limits_{n\to\infty}d_{TV}(S_n,N)\leq \big(8\pi\sup\limits_{k\geq 1}J(X_k)\big)^{\frac{1}{2}}\varepsilon.
$$
Letting $\varepsilon\to 0$ gives us the desired conclusion.\hfill$\square$

\noindent{\bf Proof of Corollary \ref{p30lu1f}.}
Let $(N_k)_{k\geq 1}$ be a sequence of independent normal random variables, and independence of $(Y_k)_{k\geq 1}.$ For each $k\geq 1,$ $\E N_k=0$ and $\E N_k^2=\E Y_k^2.$ We consider the sequence $X_k=Y_k+N_k,k\geq 1$ and define $S_n$ as in (\ref{dkrt}). For every $\varepsilon>0,$ we have
\begin{align*}
\frac{1}{b_n^2}\sum\limits_{k=1}^n\E[|X_k|^2\ind_{\{|X_k|>\varepsilon b_n\}}]&=\frac{1}{b_n^2}\sum\limits_{k=1}^n\E[|Y_k+N_k|^2\ind_{\{|Y_k+N_k|>\varepsilon b_n\}}]\\
&\leq \frac{2}{b_n^2}\sum\limits_{k=1}^n\E[(|Y_k|^2+|N_k|^2)\ind_{\{|Y_k|+|N_k|>\varepsilon b_n\}}]\\
&\leq \frac{2}{b_n^2}\sum\limits_{k=1}^n\E[(|Y_k|^2+|N_k|^2)(\ind_{\{|Y_k|>\varepsilon b_n/2\}}+\ind_{\{|N_k|>\varepsilon b_n/2\}})].
\end{align*}
By Chebyshev's association inequality we deduce
\begin{align*}
\frac{1}{b_n^2}\sum\limits_{k=1}^n\E[|X_k|^2\ind_{\{|X_k|>\varepsilon b_n\}}]&\leq \frac{4}{b_n^2}\sum\limits_{k=1}^n\E[|Y_k|^2\ind_{\{|Y_k|>\varepsilon b_n/2\}}]+\frac{4}{b_n^2}\sum\limits_{k=1}^n\E[|N_k|^2\ind_{\{|N_k|>\varepsilon b_n/2\}}].
\end{align*}
It is easy to check that $\frac{4}{b_n^2}\sum\limits_{k=1}^n\E[|N_k|^2\ind_{\{|N_k|>\varepsilon b_n/2\}}]\to 0$ as $n\to\infty.$ Hence, by the condition (\ref{lin01p}), we conclude that the sequence $(X_k)_{k\geq 1}$ also satisfies Lindeberg's condition.

On the other hand, by Lemma 1.20 in \cite{Johnson2004}, the score function of $X_k$ is given by
$$\rho_k(X_k)=\E\left[\frac{-N_k}{\alpha_k^2}\bigg|X_k\right],\,\,k\geq 1.$$
This implies that $\sup\limits_{k\geq 1}J(X_k)<\infty.$ Indeed, we have
$$J(X_k)={\rm Var}(X_k)\E[\rho_k^2(X_k)]=2\alpha_k^2\E\left(\E\left[\frac{-N_k}{\alpha_k^2}\bigg|X_k\right]\right)^2\leq 2,\,\,k\geq 1.$$
Thanks to Corollary \ref{plu1f}, we obtain $d_{TV}(S_n,N)\to 0$ as $n\to\infty.$ For every $t\in \mathbb{R},$ in view of L\'evy's continuity theorem, we have
\begin{align*}
\E\left[\exp\left(it\frac{Y_1+\cdots+Y_n}{a_n}\right)\right]&=e^{t^2/2}\E\left[\exp\left(it\frac{Y_1+N_1+\cdots+Y_n+N_n}{a_n}\right)\right]\\
&=e^{t^2/2}\E\left[\exp\left(it\sqrt{2}S_n\right)\right]\to e^{-t^2/2},\,\,n\to\infty
\end{align*}
and hence, $\frac{Y_1+\cdots+Y_n}{a_n}\to N$ in distribution.

The proof is complete.\hfill$\square$




\end{document}